\documentclass[12pt]{article}
\usepackage{amsmath,amssymb,amsfonts,latexsym}

\newenvironment{proof}{\noindent{\bf Proof.}\hspace*{1em}}{\bigskip}

\newcommand\Proof{\noindent{\it Proof.\/ }}
\newcommand\Endproof{\hfill$\Box$\medbreak}

\newtheorem{theorem}{Theorem}[section]
\newtheorem{corollary}{Corollary}[theorem]
\newtheorem{lemma}{Lemma}[section]
\newtheorem{rem}{Remark}[section]

\newtheorem{definition}{Definition}[section]
\newtheorem{defn}{Definition}[section]
\newtheorem{statement}{Proposition}[section]

\def\JC{\operatorname{JC}}

\def\ch{\operatorname{Char}}

\def\Id{\operatorname{Id}}
\def\Ad{\operatorname{Ad}}
\def\End{\operatorname{End}}
\def\Nilp{\operatorname{Solv}}
\def\Aut{\operatorname{Aut}}
\def\PI{\operatorname{PI}}
\def\Var{\operatorname{Var}}
\def\goth{\mathfrak}
\newcommand\C{{\mathbb C}}

\newcommand\A{{\mathbb A}}
\renewcommand\k{{\bf k}}
\newcommand\spec{{\mathsf{Spec}\,}}

\title{The Jacobian Conjecture, together with Specht and Burnside-type problems}

\makeatletter

\def\@footnotetext#1{\insert\footins{%
\footnotesize
    \interlinepenalty\interfootnotelinepenalty

    \splittopskip\footnotesep

    \splitmaxdepth \dp\strutbox \floatingpenalty \@MM

    \hsize\columnwidth \@parboxrestore

    \edef\@currentlabel{\csname
p@footnote\endcsname\@thefnmark}\@makefntext
    {\rule{\z@}{\footnotesep}\ignorespaces
#1\strut}}}

\def\abstract{\small\quotation{\hskip-\parindent\sc Abstract.}}
\def\classification{\@ifnextchar [{\@xfootnotenext}%
    {\begingroup\let\protect\noexpand
    \xdef\@thefnmark{}\endgroup
    \@footnotetext}}

\begin{document}

\maketitle

\classification{
$\ast$)
  The first and third authors are supported by the Israel Science Foundation grant No.~1207/12. The research of Jie-Tai Yu was partially supported by an
RGC-GRF Grant. \\  Yagzhev was a doctoral student of the second
author, L.~ Bokut. \\ We thank I.P.~Shestakov for useful comments,
and also thank  the referees for many helpful suggestions in
improving the exposition.\\
 We are grateful to Yagzhev's widow
G.I.~Yagzheva, and also Jean-Yves Sharbonel for providing some
unpublished materials.}

\def\abstract{\small\quotation{\hskip-\parindent\sc Abstract.}}
{\begingroup\let\protect\noexpand%
\xdef\@thefnmark{}\endgroup
\maketitle

\begin{center}
{\bf Alexei Belov, Leonid Bokut, Louis Rowen and Jie-Tai Yu}
\end{center}
\medskip


\begin{center}
\qquad Dedicated to the memory of A.V.~Yagzhev
\end{center}


\begin{abstract}
We explore an approach to the celebrated Jacobian Conjecture by
means of identities of algebras, initiated by the brilliant
deceased mathematician, Alexander Vladimirovich Yagzhev~
(1951--2001), only partially published. This approach also
indicates some very close connections between mathematical
physics, universal algebra and automorphisms of polynomial
algebras.
\end{abstract}



\section{Introduction}

This paper explores an approach to polynomial mappings and the
Jacobian Conjecture and related questions, initiated by
A.V.~Yagzhev, whereby these questions are translated to identities
of algebras, leading to a solution  in \cite{Yagzhev8} of the
version  of the Jacobian Conjecture for free associative algebras.
(The first version, for two generators, was obtained by Dicks and
J.~Levin \cite{Dicks,DicksLevin}, and the full version by
Schofield~\cite{Schofield}.) We start by laying out the basic
framework in this introduction. Next, we set up Yagzhev's
correspondence to algebras in \S\ref{operad}, leading to the basic
notions of weak nilpotence and Engel type. In
\S\ref{SbScJCArbVar0} we discuss the Jacobian Conjecture in the
context of various varieties, including the free associative
algebra.

Given any  polynomial endomorphism $\phi$ of the $n$-dimensional
affine space $\A^n_\k=\spec \k[x_1,\dots,x_n]$ over a field $\k$,
we define its {\it Jacobian matrix} to be the matrix
$$ \left(
\partial \phi^*(x_i)/ \partial x_j\right)_{1\le i,j\le n}.$$
The determinant of the Jacobian matrix is called  the {\it
Jacobian} of $\phi$. The celebrated {\bf Jacobian Conjecture}
$\JC_n$ in dimension $n\ge 1$ asserts that {\it for any field~$\k$
of characteristic zero, any polynomial endomorphism $\phi$ of
$\A^n_\k $ having Jacobian $1$ is an
 automorphism.} Equivalently, one can say that $\phi$ preserves the
 standard top-degree differential form $dx_1\wedge\dots\wedge
dx_n\in\Omega^n(\A^n_\k)$. References to this well known problem
and related questions can be found in~\cite{BCW}, \cite{Kulikov}, and
\cite{VDEssen}.
 By the Lefschetz principle it is sufficient to consider the case  $\k=\C$;
obviously, $\JC_n$ implies $\JC_m$ if $n>m$.
  The conjecture $\JC_n$ is obviously true  in the case
$n=1$, and it is open for $n\ge 2$.

The {\bf Jacobian Conjecture}, denoted as $\JC $, is the
conjunction of the conjectures $\JC_n$ for
 all finite $n$.  The Jacobian
Conjecture has many reformulations (such as the Kernel Conjecture
and the Image Conjecture,
cf.~\cite{VDEssen,VDEssenImage,VDEssenImageA,VDEssenWenhua,VDEssenWenhua1}
for details) and is closely related to questions concerning
quantization. It is stably equivalent to the following conjecture
of Dixmier, concerning automorphisms of the Weyl algebra $W_n$,
otherwise known as the {\it quantum affine algebra}.

\medskip
{\bf Dixmier Conjecture $DC_n$:}\ {\it Does
$\End(W_n)=\Aut(W_n)$}?
\medskip

The implication $DC_n\to JC_n$ is well known,  and the inverse
implication $JC_{2n}\to DC_n$ was recently obtained independently
by Tsuchimoto \cite{Tsuch1} (using $p$-curvature) and Belov and
Kontsevich \cite{BelovKontsevich1}, \cite{BelovKontsevich} (using
Poisson brackets on the center of the Weyl algebra).
Bavula~\cite{Bavula1new} has obtained a shorter proof, and also
obtained a positive solution of an analog of the Dixmier
Conjecture for integro differential operators, cf.~\cite{Bavula1}.
He also proved that every monomorphism of the  Lie algebra of
triangular polynomial derivations is an automorphism
\cite{Bavula2} (an analog of Dixmier's conjecture).

The Jacobian Conjecture is closely related to many questions of
affine algebraic geometry concerning affine space, such as the
Cancellation Conjecture  (see Section \ref{SbScRelQuestJC}). If we
replace the variety of commutative associative algebras (and the
accompanying affine spaces) by an arbitrary algebraic variety
\label{PgAlgVar}\footnote{Algebraic geometers use  word {\it
variety}, roughly speaking, for objects whose local structure is
obtained from the solution of system of algebraic equations. In
the framework of universal algebra, this notion is used for
subcategories of algebras defined by a  given set of identities. A
deep analog of these notions is given in \cite{BelovIAN}.}, one
easily gets a counterexample to the  $\JC$. So, strategically
these questions deal with some specific properties of affine space
which we do not yet understand, and for which we do not have the
appropriate formulation apart from these very difficult questions.

 It seems that these properties do indicate some sort of quantization. From that
perspective, noncommutative analogs of these problems (in particular,
the Jacobian Conjecture and the analog of the
Cancellation Conjecture)  become interesting for free associative
algebras, and more generally, for arbitrary  varieties of algebras.

We work in the language of universal algebra, in which an algebra
is defined in terms of a set of operators, called its {\it
signature}. This approach enhances the investigation of the
Yagzhev correspondence between endomorphisms and algebras. We work
with deformations and so-called {\it packing properties} to be
introduced in Section~\ref{SbScJCArbVar0} and
Section~\ref{SbSbScPacking}, which denote specific noncommutative
phenomena which enable one to solve   the  $\JC$ for the free
associative algebra.

 From the viewpoint of universal
algebra,  the Jacobian conjecture becomes a problem of ``Burnside
type,'' by which we mean the question of whether a given finitely
generated algebraic structure satisfying given periodicity
conditions is necessarily finite, cf.~Zelmanov~\cite{Zelmanov}.
 Burnside originally posed the question
of the finiteness of a finitely generated  group satisfying the
identity $x^n=1$. (For odd $n\ge 661,$  counterexamples were found
by Novikov and Adian, and quite recently  Adian reduced
the estimate from $661$ to $101$).  Another class of
counterexamples was discovered by Ol'shanskij~\cite{Ol}.
Kurosh posed the question of local finiteness of algebras
whose elements are algebraic over the base field. For algebraicity
of bounded degree, the question has a positive solution, but
otherwise there are the Golod-Shafarevich counterexamples.

 Burnside type
problems play an important role in algebra. Their solution in
the associative case is closely tied to Specht's
problem of whether any set of polynomial identities can be deduced
from a finite subset. The $\JC$ can be  formulated in the context of whether one system of identities implies another, which also relates to Specht's problem.

In the Lie algebra case there is a similar notion. An element
$x\in L$ is called {\it Engel of degree $n$} if
$[\dots[[y,x],x]\dots, x]=0$ for any $y$ in the Lie algebra~$L$.
Zelmanov's result that any finitely generated Lie algebra of
bounded Engel degree is nilpotent yielded his solution of the
Restricted Burnside Problem for groups. Yagzhev introduced the
notion of {\it Engelian} and {\it weakly nilpotent} algebras of
arbitrary signature (see Definitions \ref{Engtyp},
\ref{DfWeaklyNilp}), and proved that the $\JC$ is equivalent to
the question of weak nilpotence of algebras of Engel type
satisfying a system of Capelli identities, thereby showing the
relation of the  $\JC$ with problems of Burnside type.

\medskip

 {\bf A negative approach.} Let us mention
a way of constructing counterexamples. This approach, developed by
Gizatullin, Kulikov,  Shafarevich, Vitushkin, and others,
is related to
decomposing polynomial mappings into the composition of
$\sigma$-processes
\cite{Gizatullin,Kulikov,Shafarevich,Vitushkin,Vitushkin1,Vitushkin2}.
It allows one to solve some polynomial automorphism problems,
including tameness problems, the most famous of which is  {\em
Nagata's Problem} concerning the wildness of  Nagata's
automorphism
 $$  (x,y,z)\mapsto(x-2(xz+y^2)y-(xz+y^2)^2z,\, y+(xz+y^2)z,\, z),$$
cf.~\cite{Nag}. Its solution by Shestakov and Umirbaev~\cite{ShUm} is the major advance in this area in the last decade. The Nagata automorphism can be constructed as a
product of automorphisms of $K(z)[x,y]$, some of them having
non-polynomial coefficients (in $K(z)$).
 The following theorem of
Abhyankar-Moh-Suzuki \cite{AM,Su} and \cite{M} can be viewed in
this context:

\medskip
\textbf{AMS Theorem.}  If $f$ and $g$ are polynomials in $K[z]$ of
degrees $n$ and $m$ for which $K[f, g] = K[z]$, then $n$ divides
$m$ or $m$ divides $n$.\medskip

Degree estimate theorems are polynomial analogs to Liouville's
approximation theorem in algebraic number theory
(\cite{BonnetVenerau,Kuroda,YuYungChang,MLY}). T.~Kishimoto has
proposed  using a program of Sarkisov, in particular for Nagata's
Problem. Although difficulties remain in applying
$``\sigma$-processes'' (decomposition of birational mappings into
standard blow-up operations) to the affine case, these may provide
new insight. If we consider affine transformations of the plane,
we have relatively simple singularities at infinity, although for
bigger dimensions they can be more complicated. Blow-ups provide
some understanding of birational mappings with singularities.
Relevant information may be provided in the affine case.  The
paper~\cite{BEW} contains some deep
 considerations about singularities.

\section{The Jacobian Conjecture and Burnside type problems, via algebras}\label{operad}

In this section we translate the  Jacobian Conjecture to the
language of algebras and their identities. This can be done at two
levels:  At the level of the algebra obtained from a polynomial
mapping, leading to the notion of {\it weak nilpotence} and {\it
Yagzhev algebras} and at the level of the differential and the
algebra arising from the Jacobian, leading to the  notion of {\it
Engel type}. The Jacobian Conjecture   is  the link between these
two notions.

\subsection{The Yagzhev correspondence}\label{Yagcor}

\subsubsection{Polynomial mappings in universal algebra}

Yagzhev's approach is to pass from algebraic geometry to universal
algebra. Accordingly, we work in the framework of a universal
algebra $A$ having signature $\Omega$. $A ^{(m)}$ denotes $A \times
\dots \times A $, taken $m$ times.

We fix a commutative, associative base ring~$C$, and consider
$C$-modules equipped with extra {\em operators} $A ^{(m)}\to A,$
which we call $m$-{\em ary}. Often one of these operators will be
(binary) multiplication. These operators will be multilinear,
i.e., linear with respect to each argument. Thus, we can define
the {\em degree} of an operator to be its number of arguments. We
say an operator $\Psi(x_1, \dots, x_m)$ is {\em symmetric} if
$\Psi(x_1, \dots, x_m) = \Psi(x_{\pi(1)}, \dots, x_{\pi(m)})$ for
all permutations~ $\pi$.

\begin{definition}\label{DefPolMap}
A {\em string} of operators is defined inductively. Any  operator
$\Psi(x_1, \dots, x_m)$ is a string of degree $m$, and if $s_j$
are strings of degree $d_j,$ then $\Psi(s_1, \dots, s_m)$ is a
string of degree $\sum_{j=1}^m d_j$.  A mapping
 $$\alpha: A ^{(m)} \to A$$
 is called {\em polynomial} if it can be
expressed as a sum of strings of operators of the algebra $A$. The
{\em degree} of the mapping is the maximal length of these
strings. \end{definition}

{\bf Example.} Suppose an algebra $A$ has two extra operators: a
binary operator $\alpha(x,y)$ and a tertiary operator
$\beta(x,y,z)$. The mapping $F: A\to A$ given by $\ x\to
x+\alpha(x,x)+\beta(\alpha(x,x),x,x)$ is a polynomial mapping
of~$A$, having degree $4$. Note that if $A$ is finite dimensional
as a vector space, not every polynomial mapping of $A$ as an
affine space is a polynomial mapping of $A$ as an algebra.
\medskip

\subsubsection{Yagzhev's
correspondence  between polynomial mappings and algebras}

Here we associate an algebraic structure to each polynomial map.
Let $V$ be an $n$-dimensional vector space over the field $\k$,
and $F:V\to V$  be a polynomial mapping of degree $m$.
 Replacing $F$ by the composite $TF$,  where $T$ is a translation such that
$TF(0)=0$, we may assume that $F(0) =0$. Given a base
$\{\vec{e}_i\}_{i=1}^n$ of $V$, and for an element $v$ of $V$
written uniquely as a sum $\sum x_i\vec{e}_i$, for $x_i \in \k$,
the coefficients of $\vec{e}_i$ in $F(v)$ are (commutative)
polynomials in the~ $x_i$. Then $F$ can be written in the
following form:
$$
x_i\mapsto F_{0i}(\vec{x})+F_{1i}(\vec{x})+\dots+F_{mi}(\vec{x})
$$
where each $F_{\alpha i}(\vec{x})$ is a homogeneous form of degree
$\alpha$, i.e.,
$$
F_{\alpha i}(\vec{x})=\sum_{j_1 + \dots + j_n =\alpha}\kappa_J
x_1^{j_1}\cdots x_n^{j_n},
$$
with $F_{0i}=0$ for all $i$,  and $F_{1i}(\vec
x)=\sum_{k=1}^n \mu_{ki}x_k$.

 We are
interested in invertible mappings that have a nonsingular Jacobian
matrix $(\mu_{ij})$. In particular, this matrix is nondegenerate
at the origin. In this case $\det(\mu_{ij})\ne 0$, and by
composing $F$ with an affine transformation we arrive at the
situation for which $\mu_{ki}=\delta_{ki}$. Thus,  the mapping $F$
may be taken to have  the following form:

\begin{equation}   \label{Eq1Fk}
x_i\to x_i-\sum_{k=2}^m F_{ki}.
\end{equation}

Suppose we have a mapping as in \eqref{Eq1Fk}. Then the Jacobi
matrix can be written as $E - G_1 - \dots - G_{m-1}$ where $G_i$
is an $n\times n$ matrix with entries which are homogeneous
polynomials of degree $i$. If the Jacobian is $1$, then it is
invertible with inverse a polynomial matrix (of  homogeneous
degree at most $(n-1)(m-1)$, obtained via the adjoint matrix).

If we write the inverse as
a formal power series, we   compare the homogeneous components
and get:
 \begin{equation}  \label{EqMJ}
\sum_{j_im_{j_i}=s}M_J=0,
\end{equation}
where $M_J$ is the sum of products $a_{\alpha_1}
a_{\alpha_q}$ in which the factor $a_j$ occurs $m_j$~times, and
$J$ denotes the multi-index $(j_1,\dots,j_q)$.

Yagzhev considered the cubic homogeneous mapping $\vec{x}\to
\vec{x}+(\vec{x},\vec{x},\vec{x}),$ whereby the Jacobian matrix
becomes $E-G_3.$ We return to  this case in~Remark~\ref{Jac3}.
  The   slightly more general approach
 given here presents the
Yagzhev correspondence more clearly and also provides tools for
investigating deformations and packing properties (see Section
\ref{SbSbScPacking}). Thus, we consider not only the cubic case
(i.e. when the mapping has the form $$x_i\to
x_i+P_i(x_1,\dots,x_n);\ i=1,\dots, n,$$ with $P_i$
 cubic homogenous polynomials), but the more general situation of arbitrary degree.

For any $\ell$, the set of (vector valued) forms
$\{F_{\ell,i}\}_{i=1}^n$ can be interpreted as a homogeneous
mapping $\Phi_\ell: V\to V$ of degree $\ell$.  When  $\ch(\k)$
does not divide $\ell$, we take instead the polarization of this
mapping, i.e.~the multilinear symmetric mapping
$$
\Psi_\ell:V^{\otimes\ell}\to V
$$
 such that
$$
(F_{\ell ,i}(x_1),\dots,F_{\ell,
i}(x_n))=\Psi_\ell(\vec{x},\dots,\vec{x}) \cdot \ell!
$$
Then Equation (\ref{Eq1Fk}) can be rewritten as

\begin{equation}   \label{Eq2Fk}
\vec{x}\to \vec{x}-\sum_{\ell=2}^m
\Psi_{\ell}(\vec{x},\dots,\vec{x}).
\end{equation}

We define the  algebra  $(A, \{ \Psi_{ \ell}\}),$ where $A$ is the
vector space $V$ and the~$ \Psi_{ \ell}$ are viewed as operators
$A^{\ell}\to A.$

\begin{definition}
 The \textbf{Yagzhev correspondence} is the correspondence from
 the polynomial mapping $(V,F)$ to the   algebra  $(A, \{ \Psi_{ \ell}\}).$
\end{definition}

\subsection{Translation of the invertibility condition to the language of
identities}

The next step is to bring in algebraic varities, defined in terms of identities.

\begin{definition}
 A {\em polynomial identity} (PI) of $A$ is a
polynomial mapping of~$A,$ all of whose values are identically
zero.

The {\it algebraic variety} generated by an algebra $A$, denoted
as $\Var(A)$, is the class of all algebras satisfying the same PIs
as $A$.
\end{definition}

Now we come to a crucial idea of Yagzhev:

\medskip
{\it The invertibility of $F$ and the invertibility of the
Jacobian of $F$ can be expressed via (2)  in the language of polynomial
identities.}
\medskip

%

Namely, let $y=F(x)=x-\sum_{\ell =2}^m\Psi_\ell (x)$. Then

\begin{equation}   \label{EqTerm}
F^{-1}(x)=\sum_t t(x),
\end{equation}

 \noindent where each $t$ is a {\it
term}, a formal expression in the mappings $\{\Psi_\ell \}_{\ell
=2}^m$ and the symbol $x$. Note that the  expressions
$\Psi_2(x,\Psi_3(x,x,x))$ and $\Psi_2(\Psi_3(x,x,x),x)$ are
different although they represent same element of the algebra.
Denote by $|t|$  the number of occurrences of variables, including
multiplicity, which are included in $t$.

The invertibility of $F$ means that, for all $q\ge q_0$,
\begin{equation}    \label{EqTerm2}
\sum_{|t|=q}t(a)=0, \quad \forall a \in A.
\end{equation}

Thus we have translated  invertibility of the mapping $F$ to the
language of identities. (Yagzhev had an analogous formula, where
the terms only involved $\Psi_3$.)
%

\begin{definition}     \label{DfWeaklyNilp}
  An element $a\in A$ is called {\em nilpotent} of   index $\le n$
if $$M(a, a, \dots, a)= 0$$  for each monomial $M(x_1, x_2, \dots
)$ of degree $\ge n$.   $A$~is {\em weakly nilpotent} if each
element of $A$ is nilpotent.  $A$~is {\em weakly nilpotent of
class~$k$} if each element of $A$ is nilpotent of index $k$. (Some
authors use the terminology {\it index} instead of {\em class}.)
Equation \eqref{EqTerm2} means $A$ is weakly nilpotent.

To stress this fundamental notion of Yagzhev, we define a
 {\em Yagzhev algebra} of   {\em order} $q_0$ to be a weakly nilpotent algebra, i.e.,
 satisfying the
identities  \eqref{EqTerm2}, also called the {\it system of
Yagzhev identities} arising from $F$.
\end{definition}

Summarizing, we get the following fundamental translation from
conditions on the endomorphism $F$ to identities of algebras.

\begin{theorem}
The endomorphism $F$ is invertible if and only if the
corresponding algebra  is a Yagzhev algebra of high enough order.
\end{theorem}

\subsubsection{Algebras of Engel type}

The analogous procedure can be carried out for the differential mapping.
We recall that $\Psi_\ell$ is a symmetric  multilinear mapping of
degree $\ell$.  We denote the mapping $y\to\Psi_\ell (y,x,\dots,x)$ as
$\Ad_{\ell-1}(x)$.

\begin{definition}\label{Engtyp}
An algebra  $A$ is of {\em Engel type}  $s$ if it satisfies a system of
identities
\begin{equation}  \label{EqMJ2}
\sum_{\ell m_\ell =s}  \quad \sum_ {\alpha _{1}+ \cdots + \alpha
_{q}=m_\ell }\Ad_{\alpha_1}(x)\cdots \Ad_{\alpha_q}(x)= 0.
\end{equation}
$A$  is of {\em Engel type}    if  $A$   has Engel type $s$  for
some $s$.
\end{definition}

.

\begin{theorem}\label{corresconj}
The endomorphism $F$ has Jacobian $1$ if and only if the
corresponding algebra has Engel type $s$  for some $s$.
\end{theorem}
\begin{proof}
Let $x'=x+dx$. Then
\begin{equation}\begin{aligned}
\Psi_\ell (x')= & \Psi_\ell (x)+\ell \Psi_\ell (dx,x,\dots,x)
\\ \qquad \qquad + & \quad \mbox{forms\ containing \ more\ than\ one occurence\ of }
dx.\end{aligned}
\end{equation}

Hence the differential of the mapping
$$
F:\vec{x}\mapsto \vec{x}-\sum_{\ell =2}^m \Psi_\ell (\vec{x}, \dots, \vec{x})
$$
is
$$
\left(E-\sum_{\ell =2}^m \ell \Ad_{\ell-1} (x)\right)\cdot dx
$$
The identities (\ref{EqMJ}) are equivalent to the   system
of identities  (\ref{EqMJ2}) in the signature $\Omega=(\Psi_2,\dots,\Psi_m)$,
 taking $a_{\alpha_j} = \Ad_{\alpha_j}$ and $m_j = \deg \Psi_\ell -1$.


\end{proof}

Thus, we have reformulated the condition of invertibility of the
Jacobian in the language of identities.

As explained in \cite{VDEssen}, it is well known from \cite{BCW}
and \cite{Yagzhev3}   that the Jacobian Conjecture can be reduced
to the cubic homogeneous case; i.e., it is enough to consider
mappings of type
$$
x\to x+\Psi_3(x,x,x).
$$
In this case the Jacobian assumption is equivalent to the {\it
Engel condition} ~-- nilpotence of the mapping $\Ad_3(x)[y]$ (i.e.
the mapping $y\to (y,x,x)$). Invertibility, considered in
\cite{BCW}, is equivalent to
  weak nilpotence, i.e., to the identity $\sum_{|t|=k}t=0$
   holding for all sufficiently large $k$.

\medskip
\begin{rem}\label{Jac3}  In the cubic homogeneous case, $j=1$, $\alpha _j=2$ and $m_j = s$,
and we define the linear map
$$\Ad_{xx}: y\to (x,x,y)$$ and the index set $T_j\subset
\{1,\dots,q\}$ such that $i\in T_j$ if and only if $\alpha_i=j$.

Then the equality (\ref{EqMJ2}) has  the following form:
$$
\Ad_{xx}^{s/2}=0.
$$
 Thus, for a ternary symmetric algebra, Engel type
means that the operators $\Ad_{xx}$ for all $x$ are nilpotent. In
other words,  the mapping $$Ad_3(x): y\to (x,x,y)$$ is nilpotent.
Yagzhev called this the  {\em Engel condition}. (For Lie algebras
the nilpotence of the operator $\Ad_x: y\to (x,y)$ is the usual
Engel condition. Here we have a generalization for arbitrary
signature.)

Here  are Yagzhev's original  definitions, for edification. A binary algebra $A$ is
 {\em Engelian} if for any element $a\in A$ the subalgebra
$<\!R_a,L_a\!>$ of vector space endomorphisms of $A$  generated by
 the left multiplication operator $L_a$ and the right
 multiplication operator
~$R_a$ is nilpotent, and  {\em weakly Engelian} if for any element
$a\in A$ the operator $R_a+L_a$ is nilpotent. 
\end{rem}

This leads us to the {\it Generalized Jacobian Conjecture}:

\medskip
{\bf Conjecture.} {\it Let $A$ be an algebra with symmetric
$\k$-linear operators $\Psi_\ell$, for $ \ell =1,\dots,m$. In any
variety of
 Engel type, $A$ is a Yagzhev algebra. }
\medskip

By Theorem~\ref{corresconj}, this conjecture would yield  the Jacobian
Conjecture.

\subsubsection{The case of binary algebras }
When $A$ is a binary algebra, {\it Engel type} means that
the left and right multiplication mappings are both nilpotent.

 A well-known result of
S.~Wang~\cite{BCW} shows that the Jacobian Conjecture holds for
quadratic mappings
$$
\vec{x}\to \vec{x}+\Psi_2(\vec{x},\vec{x}).
$$

If  two different points $(x_1,\dots,x_n)$ and $(y_1,\dots,y_n)$
of an affine space  are mapped to the same point by
$(f_1,\dots,f_n)$, then the fact that the vertex of a parabola is
in the middle of the interval whose endpoints are at the roots
shows that all $f_i(\vec{x})$ have gradients at this midpoint
$P=(\vec{x}+\vec{y})/2$ perpendicular to the line segment
$[\vec{x},\vec{y}]$. Hence   the  Jacobian is zero at the midpoint
$P$. This fact holds in any characteristic $\ne 2$.

In Section~\ref{LiftY}  we prove the following theorem of Yagzhev,
cf.~Definition~\ref{Cap1} below:

\begin{theorem}[Yagzhev]\label{Yag2}
Every symmetric binary Engel type algebra of order $k$  satisfying
the
  system of Capelli identities  of order $n$  is
weakly nilpotent, of weak nilpotence index  bounded by some
function $F(k,n)$.
\end{theorem}

\begin{rem} Yagzhev 
formulates his theorem in the following way:\medskip

 {\it Every binary weakly Engel
algebra of order $k$ satisfying   the system of Capelli identities
of order $n$   is weakly nilpotent, of   index bounded by some
function $F(k,n)$}.

\medskip
We obtain this  reformulation, by replacing the algebra $A$ by the
algebra~$A^+$ with multiplication given by $(a,b)=ab+ba$.
\end{rem}

The following problems may help us understand the situation:

\medskip
{\bf Problem.}\ {\it Obtain a straightforward proof of this
theorem and deduce from it the Jacobian
 Conjecture for quadratic
mappings. }
\medskip

{\bf Problem. (Generalized Jacobian Conjecture for quadratic
mappings)}\ {\it Is every symmetric binary algebra  of Engel type
 $k$, a Yagzhev algebra?}
\medskip

\subsubsection{The case of ternary algebras}

As we have observed, Yagzhev reduced the Jacobian Conjecture over a field of
characteristic zero to the question:

\medskip
{\it Is every  finite dimensional ternary Engel algebra  a Yagzhev algebra? }
\medskip

\noindent  Dru\'zkowski \cite{Dru1,Dru2} reduced this to the case
when all cubic forms $\Psi_{3i}$ are cubes of linear forms. Van
den Essen and his school reduced the  $\JC$ to the symmetric case;
see \cite{VDEssenBondt,VDEssenBondt1} for details.
  Bass,  Connell, and  Wright \cite{BCW} use other methods including inversions. Yagzhev's approach matches that of
\cite{BCW}, but using identities instead.   

\subsubsection{An example  in nonzero characteristic of an Engel algebra that is not  a Yagzhev algebra}
\label{SbScEngNoWNilpCharpos}

Now we give an example,   over an arbitrary field
$\k$ of characteristic $p>3$, of a finite  dimensional Engel
algebra that is not  a Yagzhev algebra, i.e., not weakly nilpotent.
This means that the situation for binary algebras
   differs intrinsically from that for ternary algebras,
and it would be worthwhile to understand why.

\begin{theorem}      \label{ThEngWeakNoNilp}
If\ $\ch(\k)=p>3$, then there exists a finite dimensional $\k$-algebra that is Engel  but not weakly
nilpotent.
\end{theorem}
\Proof
 Consider the noninvertible mapping $F: \k[x]\to \k[x]$ with
Jacobian~1:
$$
F: x\to x+x^p.
$$
We introduce  new commuting indeterminates $\{y_i\}_{i=1}^n$ and
extend this mapping to $k[x,y_1,\dots,y_n]$ by sending $y_i\mapsto
y_i$. If $n$ is big enough, then it is possible to find  tame
automorphisms $G_1$ and $G_2$ such that $G_1\circ F\circ G_2$ is a
cubic mapping $\vec{x}\to \vec{x}+\Psi_3(\vec{x})$, as follows:

Suppose we have a mapping
$$
F:x_i\to P(x)+M
$$
where $M=t_1t_2t_3t_4$ is a monomial of degree at least $4$.
Introduce two new commuting indeterminates $z, y$ and take $F(z)=z$,
$F(y)=y$.

Define the mapping $G_1$ via $G_1(z)=z+t_1t_2$, $G_1(y)=y+t_3t_4$
with $G_1$ fixing all other indeterminates; define $G_2$ via
$G_2(x)=x-yz$ with $G_2$ fixing all other
indeterminates.

The composite mapping $G_1\circ F\circ G_2$ sends $x$ to
$P(x)-yz-yt_1t_2-zt_3t_4$, $y$ to $y+t_3t_4$, $z$ to $z+t_1t_2$,
and agrees with $F$ on all other indeterminates.

Note that we have removed the  monomial $M=t_1t_2t_3t_4$ from the
image of $F$, but instead have obtained various monomials of
smaller degree ($t_1t_2$ , $t_3t_4$, $zy$, $zt_3t_4$, $yt_1t_2$).
It is easy to see that this process terminates.

Our new mapping $H(x)=x+\Psi_2(x)+\Psi_3(x)$ is noninvertible and
has Jacobian~1. Consider its blowup
$$R: x\mapsto x+T^2y+T\Psi_2(x),\  y\mapsto y-\Psi_3(x),\  T\mapsto T.$$


This mapping $R$ is invertible if and only if the initial mapping
is invertible, and has   Jacobian~1 if and only if the initial
mapping has Jacobian~1, by~\cite[Lemma 2]{Yagzhev3}. This mapping
is also cubic homogeneous. The corresponding ternary algebra is
Engel, but not weakly nilpotent. \Endproof

This example shows that a direct combinatorial approach to the
Jacobian Conjecture encounters difficulties, and in working with
related Burnside type problems (in the sense of Zelmanov
\cite{Zelmanov}, dealing with  nilpotence properties of Engel
algebras, as indicated in the introduction), one should take into
account specific properties arising in characteristic zero.

\begin{defn}
An algebra $ A$ is  {\em nilpotent} of   {\em class} $\le n$ if
$M(a_1, a_2, \dots )= 0$  for each monomial $M(x_1, x_2, \dots )$
of degree $\ge n$. An  ideal $I$ of $ A$ is  {\em strongly
nilpotent} of {\em class} $\le n$ if $M(a_1, a_2, \dots )= 0$  for
each monomial $M(x_1, x_2, \dots )$  in which indeterminates of
total degree $\ge n$ have been substituted to elements of $I$.
\end{defn}

 Although the notions of nilpotent and  strongly nilpotent coincide in the associative case,
they differ for ideals of nonassociative algebras.
For example, consider the following algebra suggested by Shestakov:
$A$ is the algebra generated by $a,b,z$ satisfying the relations $a^2 =b,$ $ bz = a$
and all other products 0. Then $I = Fa+Fb$ is nilpotent as a subalgebra,
 satisfying $I^3 = 0$ but not strongly nilpotent (as an ideal), since
$$b = ((a(bz))z)a \ne 0,$$
and one can continue indefinitely in this vein.
 Also, \cite{KuS} contains an example of a finite dimensional non-associative algebra without any ideal which is maximal witih respect to being nilpotent as a subalgebra.

 In connection with the Generalized Jacobian Conjecture
in characteristic~0, it follows  from results of Yagzhev
~\cite{Yagzhev9}, also cf.~\cite{GorniZampieri},  that there
exists a $20$-dimensional Engel algebra over $\mathbb Q$, not
weakly nilpotent, satisfying the identities
\begin{align*}
x^2y=-yx^2, \quad (((yx^2)x^2)x^2)x^2=0,\cr (xy+yx)y=2y^2x, \quad
x^2y^2=0.
\end{align*}

  However, this algebra can be seen to be
Yagzhev (see Definition \ref{DfWeaklyNilp}).

 For associative
algebras, one uses the term ``nil'' instead of ``weakly
nilpotent.'' Any nil subalgebra of a
  finite  dimensional associative algebra is nilpotent, by
Wedderburn's Theorem \cite{Wedderburn}). Jacobson generalized this
result to other settings, cf.~\cite[Theorem 15.23]{Row}, and
Shestakov~\cite{ShestakovWedderburn} generalized it to a wide class of
  Jordan algebras (not necessarily commutative).

\medskip
Yagzhev's investigation of weak nilpotence has applications to the
Koethe Conjecture,  for algebras over uncountable fields.  He
reproved:

\medskip
{\bf *} {\it In every associative algebra over an uncountable
field, the sum of every two nil right  ideals is a nil right  ideal}
\cite{yagzhevKoethe}.
\medskip

(This was proved first by Amitsur~\cite{Am}. Amitsur's result is
for affine algebras, but one can easily reduce to the affine
case.) 
%
%

\subsubsection{Algebras satisfying systems of Capelli identities}\label{SbSbScCap}

\begin{definition}\label{Cap1} The
 {\it Capelli polynomial} $C_k$ of order $k$ is
 $$C_k := \sum _{\sigma\in S_k } (-1)^{\sigma}x_{\sigma(1)}y_1 \cdots
 x_{\sigma(k)}y_k.$$
 \end{definition}

  It is obvious that an associative algebra
satisfies the  Capelli identity  $c_k$ iff, for any monomial
$M(x_{1},\ldots,x_{k},y_{1},\ldots,y_{r})$ multilinear in the
$x_i$, the following equation holds  identically in $A$:
\begin{equation}        \label{EquVarStar}
\sum_{\sigma\in S_k }(-1)^{\sigma} M(v_{\sigma(1)},\ldots,
  v_{\sigma(k)},
  y_{1},\ldots,y_{r}) = 0.
\end{equation}


However, this does not apply to nonassociative algebras, so we
need to generalize this condition.

\begin{definition} The algebra $A$ satisfies a
 {\bf system of Capelli identities} of order $k$, if \eqref{EquVarStar}
holds identically in $A$ for any monomial
$M(x_{1},\ldots,x_{k},y_{1},\ldots,y_{r})$ multilinear in the
$x_i$. \end{definition}

Any algebra of dimension $<k$ over a field satisfies a
 system of Capelli identities  of order $k$.
%
%
Algebras satisfying systems of Capelli identities behave much like
finite dimensional algebras. They were introduced and
systematically studied by  Rasmyslov \cite{RazmyslovIANCapely},
\cite{RazmyslovBook}.

 Using Rasmyslov's method,
Zubrilin~\cite{Zubrilin4}, also see \cite{RazmZubrilin,Zubrilin1},
proved that if $A$ is an arbitrary algebra satisfying the system
of Capelli identities of order~$n$, then the chain of ideals
defining the {\it solvable radical} stabilizes at the $n$-th step.
More precisely, we utilize a Baer-type radical, along the lines of
Amitsur~\cite{am2}.

Given an algebra $A$, we define $\Nilp_1: = \Nilp_1(A)  = \sum
\{$Strongly nilpotent  ideals of $A\},$ and inductively, given $\Nilp_k $,
define $\Nilp_{k+1} $ by $\Nilp_{k+1}/\Nilp_k =
\Nilp_1(A/\Nilp_k).$ For a limit ordinal $\alpha,$ define
$$\Nilp_\alpha  = \cup _{\beta < \alpha} \Nilp_\beta.$$ This must
stabilize at some ordinal $\alpha$, for which we define $\Nilp(A)
= \Nilp_\alpha.$

  \medskip

Clearly $\Nilp(A/\Nilp(A)) = 0;$ i.e., $A/\Nilp(A)$ has no nonzero strongly
nilpotent ideals.   Actually, Amitsur~\cite{am2}
defines $\zeta(A)$ as built up from ideals having trivial multiplication, and proves \cite[Theorem~1.1]{am2} that
$\zeta(A)$ is the intersection of the prime ideals of $A$.

We shall use the notion of {\it sandwich}, introduced by Kaplansky
and Kostrikin, which is a powerful tool for Burnside type problems
\cite{Zelmanov}. An ideal $I$ is called a {\it sandwich ideal} if,
for any $k$,
$$M(z_1,z_2,x_1,\dots,x_k)=0$$
for any $z_1, z_2\in I$, any set of elements $x_1,\dots,x_k$, and
any multilinear monomial $M$ of degree $k+2$. (Similarly, if the
operations of an algebra have degree $\le \ell$, then it  is
natural to use $\ell$-sandwiches, which by definition satisfy the
property that $$M(z_1,\dots,z_\ell,x_1,\dots,x_k)=0$$ for any
$z_1, \dots,z_\ell \in I$, any set of elements $x_1,\dots,x_k$,
and any multilinear monomial $M$ of degree $k+\ell$.)

The next useful lemma follows from a result from \cite{Zubrilin4}:

\begin{lemma}     \label{LeSandvich}
If  an ideal $I$ is strongly  nilpotent of class  $\ell$, then there exists
a decreasing sequence of ideals $I=I_1\supseteq\dots\supseteq
I_{l+1}=0$ such that $I_s/I_{s+1}$ is a sandwich ideal in
$A/I_{s+1}$ for all $s\le l$.
\end{lemma}

\begin{definition}     \label{DfRepresAlg}
An algebra $A$ is {\em representable}  if it can be embedded into
an algebra finite dimensional over some extension of the ground
field.
\end{definition}

\begin{rem} \ Zubrilin \cite{Zubrilin4}, properly clarified, proved the more precise
statement, that if an algebra $A$ of arbitrary signature satisfies
a system of Capelli identities $C_{n+1}$, then there exists a
sequence $B_0\subseteq B_1\subseteq\dots\subseteq B_n$  of
strongly nilpotent ideals  such that:
\begin{itemize}
  \item The natural projection of $B_i$ in $A/B_{i-1}$ is a strongly nilpotent ideal.

  \item $A/B_n$ is representable.
  \item If $I_1\subseteq I_2\subseteq\dots\subseteq I_n$ is any sequence of
ideals of $A$ such that $I_{j+1}/I_{j}$ is a sandwich ideal in
$A/I_j$, then $B_n\supseteq I_n$.
\end{itemize}

 Such a sequence of ideals will be called a {\em Baer-Amitsur
  sequence}.
 In affine space the Zariski closure of the radical is radical, and hence the factor algebra is representable.
 (Although the radical  coincides with the linear closure if the base field is infinite (see \cite{BelovUzyRowen2}),
 this assertion holds for arbitrary  signatures and base fields.) Hence in representable algebras,
 the Baer-Amitsur sequence stabilizes after finitely many steps.
  Lemma \ref{LeSandvich} follows from these
considerations.
\end{rem}

Our next main goal is to prove Theorem \ref{ThSandvich} below, but first we need another notion.
\subsubsection{The tree associated to a monomial}         \label{SbSbScTreeMon}

 Effects of nilpotence have been used   by different authors in
another language.
We  associate a {\it rooted labelled tree} to  any monomial: Any
branching vertex indicates the symbol of an operator, whose
outgoing edges are the terms in the corresponding symbol. Here is
the precise definition.

\begin{definition}     \label{DfTreeMon}
Let $M(x_1,\dots,x_n)$ be a monomial in an algebra $A$ of
arbitrary signature. One can associate the tree $T_M$ by an
inductive procedure:

\begin{itemize}
  \item If $M$ is a single variable, then $T_M$ is just the vertex $\bullet$.
  \item Let $M=g(M_1,\dots,M_k)$, where $g$ is a $k$-ary operator.
 We assume  inductively    that the trees $T_i,$ $ i=1,\dots, k,$ are already defined.
Then the tree $T_M$ is the disjoint union of the $T_i$, together
with the root $\bullet$ and arrows starting with $\bullet$ and
ending with the roots of the trees $T_i$.
\end{itemize}

\end{definition}

{\bf Remark.} Sometimes one  labels $T_M$ according to the
operator $g$ and the positions inside $g$.

\medskip

If the outgoing degree of each vertex is 0 or 2, the tree is
called {\it binary}. If the outgoing degree of each vertex is
either 0 or 3, the tree is called {\it ternary}. If each operator
is binary, $T_M$ will be binary; if each operator is ternary,
$T_M$ will be ternary.

\subsection{Lifting Yagzhev algebras}\label{LiftY}

 Recall Definitions \ref{DfWeaklyNilp}  and   \ref{Engtyp}.

\begin{theorem}         \label{ThSandvich}
Suppose $A$ is an algebra of   Engel type, and let $I$ be a sandwich ideal of
$A$. If $A/I$ is Yagzhev, then $A$ is Yagzhev.
\end{theorem}

{\bf Proof.} The proof follows easily from the following two
propositions.

\medskip

  Let $k$ be the class of weak  nilpotence of $A/I$. We call a
branch of the tree {\it fat} if it has more than $k$ entries.

\begin{statement}
a) The sum of all monomials of any degree $s>k$ belongs to $I$.

b) Let $x_1,\dots, x_n$ be  fixed indeterminates, and $M$ be an
arbitrary monomial, with $s_1,\dots,s_\ell>k$. Then

\begin{equation}      \label{EqSnvch}
\sum_{|t_1|=s_1,\dots, |t_\ell|=s_\ell} M(x_1,\dots,x_n,t_1,\dots, t_\ell)\equiv 0.
\end{equation}

c)  The sum of all monomials of degree $s$, containing at least $\ell$
non-intersecting fat branches, is zero.
\end{statement}

{\bf Proof.} a) is just a reformulation of the weak  nilpotence
 of~$A/I$; b) follows from a) and the sandwich property
of an ideal $I$. To get c) from~b),  it is enough to consider the
highest non-intersecting fat branches.

\begin{statement}[Yagzhev]
The linearization of the sum of all terms with a fixed fat branch
of length $n$ is the complete linearization of the function
$$
\sum_{\sigma\in S_n}\prod(\Ad_{k_{\sigma(i)}})(z)(t).
$$
\end{statement}


Theorem 1.2, Lemma \ref{LeSandvich}, and  Zubrilin's  result
give us the following major result:

\begin{theorem}\label{eqcond}
In characteristic zero, the Jacobian conjecture is equivalent to
the following statement:

Any algebra of Engel type satisfying some system of Capelli
identities is a Yagzhev algebra.
\end{theorem}

This theorem generalizes the following result of Yagzhev:

\begin{theorem}
The Jacobian conjecture is equivalent to the following statement:

Any ternary Engel algebra in characteristic 0 satisfying a system
of Capelli identities is a Yagzhev algebra.
\end{theorem}

The Yagzhev correspondence and the results of this section
  (in particular, Theorem~\ref{eqcond}) yield the proof of
  Theorem~\ref{Yag2}.

\subsubsection{Sparse identities}

Generalizing  Capelli identities, we say that an algebra satisfies
a system of {\it sparse identities} when there exist $k$ and
coefficients $\alpha_{\sigma}$ such that for any monomial
$M(x_{1},\ldots,x_{k},y_{1},\ldots,y_{r})$ multilinear in $x_i$
the following equation holds:
\begin{equation}        \label{EquVarStarGM}
\sum_{\sigma }\alpha_{\sigma} M(c_{1}v_{\sigma(1)}d_{1},\ldots,
  c_{k}v_{\sigma(k)}d_{k},
  y_{1},\ldots,y_{r}) = 0.
\end{equation}

Note that one need only check  \eqref{EquVarStarGM} for monomials.
The system of Capelli identities is a special case of a system of
sparse identities (when $\alpha_{\sigma}=(-1)^\sigma$). This
concept ties in with the following ``few long branches'' lemma
\cite{Zubrilin2}, concerning the structure of trees of monomials
for algebras with sparse identities:

\begin{lemma}[Few long branches]\label{spars}
Suppose an algebra $A$ satisfies a system of sparse identities of
order $m$. Then any monomial is linearly representable by monomials such that
the corresponding tree has not more than  $m-1$ disjoint
branches of length  $\ge m$.

\end{lemma}

 Lemma~\ref{spars} may   be useful in studying nilpotence of Engel algebras.
%
%

\subsection{Inversion formulas and problems of  Burnside type}
We have seen that the $\JC$  relates to problems of ``Specht type''
(concerning whether one set of polynomial identities implies
another), as well as problems of Burnside type.

Burnside type problems become more complicated in  nonzero
characteristic; cf.~Zelmanov's review article \cite{Zelmanov}.

  Bass, Connell, and Wright~\cite{BCW} attacked the $\JC$ by means
  of inversion formulas.
 D.~Wright \cite{Wright}
wrote an inversion formula for the symmetric case and related it
to a combinatorial structure called the {\it Grossman--Larson
Algebra}. Namely, write $F=X-H$, and define $J(H)$ to be the
Jacobian matrix of $H$.  Wright proved the $\JC$ for the case
where $H$ is homogeneous and $J(H)^3=0$, and also for the case
where $H$ is cubic and $J(H)^4=0$; these correspond in Yagzhev's
terminology to the cases of Engel type $3$ and $4$, respectively.
Also, the so-called {\it chain vanishing theorem} in~\cite{Wright}
follows from  Engel type. Similar results were obtained earlier by
Singer~\cite{Singer} using tree formulas for formal inverses. The
inversion formula, introduced in \cite{BCW}, was investigated by
D.~Wright and his school.  Many authors  use the language of
so-called  {\it tree expansion} (see \cite{Wright,Singer} for
details). In view of Theorem \ref{ThEngWeakNoNilp},  the tree
expansion technique  should be highly nontrivial.

 The Jacobian Conjecture can be formulated
as a question of quantum field theory (see \cite{Abdelmalek}), in which tree expansions are seen to correspond to Feynmann diagrams.

In the papers \cite{Singer} and \cite{Wright} (see also
\cite{Wright1}),  trees with one label correspond  to elements of
the algebra $A$  built by Yagzhev, and $2$-labelled trees
correspond to the elements of the operator algebra $D(A)$ (the
algebra generated by operators $x\to M(x,\vec{y})$, where $M$ is
some monomial). These authors deduce weak nilpotence from the
Engel conditions of degree 3 and 4.   The inversion formula for
automorphisms of tensor product of Weyl algebras and the ring of
polynomials was studied intensively in the papers
\cite{Bavula,Bavula1new}. Using  techniques from
\cite{BelovKontsevich1}, this yields a slightly different proof of
the equivalence between the  $\JC$ and DC, by an argument similar
to one given in \cite{YagzhevLast}. Yagzhev's approach makes the
situation much clearer, and the known approaches to the Jacobian
Conjecture using inversion formulas can be explained from this
viewpoint.

\begin{rem}\label{Jac4} The most
recent inversion formula (and probably the most algebraically
explicit one) was obtained by V.~Bavula \cite{Bavula3}. The
coefficient $q_0$ can be made explicit in \eqref{EqTerm2}, by
means of  the Gabber Inequality, which says that if
$$f: K^n\to K_n; \quad x_i\to f_i(\vec x)$$ is a polynomial
automorphism, with
 $\deg(f)=\max_i\deg(f_i)$, then $\deg(f^{-1})\le \deg(f)^{n-1}$)
\end{rem}

In fact, we are working with {\it operads}, cf.~the classical
book~\cite{MSS}. A review of operad theory and its relation with
physics and $\PI$-theory in particular Burnside type problems,
will appear in D.~Piontkovsky \cite{Piontkovski}; see also
\cite{Piontkovski1,PiontkovskiKhoroshkin}.  Operad theory provides
a supply of   natural identities and varieties, but they also correspond  to geometric facts. For example, the Jacobi
identity corresponds to the fact that the altitudes of a triangle
are concurrent. M.~Dehn's observations that the Desargue property
of a projective plane corresponds to associativity of its
coordinate ring, and Pappus' property to its commutativity, can be
considered as a first step in operad theory. Operads are
important in mathematical physics, and formulas for the famous
Kontsevich quantization theorem resemble formulas for the inverse
mapping. The operators considered here are operads.


\section{The Jacobian Conjecture for varieties, and deformations}\label{SbScJCArbVar0}

In this section we consider   analogs of the  $\JC$ for other
varieties of algebras, partially with the aim on throwing light on
the classical  $\JC$ (for the commutative associative polynomial
algebra).

\subsection{Generalization of the Jacobian Conjecture to arbitrary
varieties}                  \label{SbScJCArbVar}

J.~Birman~\cite{Bir} already proved the $\JC$ for free groups in
1973.
 The $\JC$ for free associative algebras (in two
generators) was established   in~1982 by~W.~Dicks and J.~Levin
\cite{DicksLevin,Dicks}, utilizing Fox derivatives, which we
describe later on. Their result was reproved by
Yagzhev~\cite{Yagzhev1}, whose  ideas are sketched in this
section. Also see  Schofield~\cite{Schofield}, who proved  the
full version. Yagzhev then applied these ideas to other varieties
of algebras \cite{Yagzhev8,Yagzhev9} including nonassociative
commutative algebras and   anti-commutative algebras;
U.U.~Umirbaev \cite{Umirbaev} generalized these to
``Schreier varieties,'' defined by the property that every subalgebra of a free algebra is free. 
The $\JC$ for free Lie algebras was proved by
Reutenauer~\cite{Reu}, Shpilrain~\cite{Shp}, and Umirbaev
\cite{Umi}.

The Jacobian Conjecture for varieties generated by finite
dimensional algebras, is closely related to the Jacobian
Conjecture in the usual commutative associative  case, which is
the most important.

Let $\goth M$ be a variety of algebras of some signature~$\Omega$
over a given field~$\k$ of characteristic zero, and $\k_{\goth
M}\!<\vec{x}\!>$~the relatively free algebra in $\goth M$ with
generators $\vec{x}=\{x_i :i\in I\}$. We assume that $|\Omega|,
|I|<\infty$, $I=1,\dots, n$.

Take a set $\vec{y}=\{y_i\}_{i=1}^n$ of new indeterminates. For
any $f(\vec{x})\in k_{\goth M}\!<\!\vec{x}\!>$ one can define an
element $\hat{f}(\vec{x},\vec{y})\in k_{\goth
M}\!<\!\vec{x},\vec{y}\!>$ via the equation

\begin{equation}         \label{EqDerFn}
f(x_1+y_1,\dots,x_n+y_n)=f(\vec{y})+\hat{f}(\vec{x},\vec{y})+
R(\vec{x},\vec{y})
\end{equation}
\noindent where $\hat{f}(\vec{x},\vec{y})$ has degree $1$ with respect
to $\vec x$, and $R(\vec{x},\vec{y})$ is the sum of monomials
of degree  $\ge 2$  with respect to $\vec x$; $\hat{f}$ is a
generalization of the differential.

Let $\alpha\in\End(k_{\goth M}\!<\!\vec{x}\!>)$, i.e.,
\begin{equation}      \label{EqEnd}
\alpha: x_i\mapsto f_i(\vec{x});\ i=1,\dots,n.
\end{equation}

\begin{definition}\label{Jac1}
Define the Jacobi endomorphism $\hat{\alpha}\in\End(\k_{\goth
M}\!<\!\vec{x},\vec{y}\!>)$ via the equality
\begin{equation}       \label{EqDif}
\hat{\alpha}:
\left\{
\begin{array}{cc}
x_i\to\hat{f}_i(\vec{x}),\cr y_i\to y_i.
\end{array}
\right.
\end{equation}
\end{definition}

The Jacobi mapping $f \mapsto \hat{ f}$ satisfies the chain rule,
in the sense that it  preserves composition.

\begin{rem}\label{chain}
 It is not
difficult to check (and is well known) that if
$\alpha\in\Aut(\k_{\goth M}\!<\!\vec{x}\!>)$ then
$\hat{\alpha}\in\Aut(\k_{\goth M}\!<\!\vec{x},\vec{y}\!>)$.
\end{rem}

The inverse implication is called the {\it Jacobian Conjecture for
the variety~$\goth M$}. Here is an important special case.

\begin{definition}
Let $A\in{\goth M}$ be a finite dimensional algebra, with base
$\{\vec{e}_i\}_{i=1}^N$. Consider a set of commutative indeterminates
$\vec{\nu}=\{\nu_{si}|s=1,\dots,n; i=1,\dots,N\}$. The elements
$$z_j=\sum_{i=1}^N\nu_{ji} \vec{e}_i; \quad  j=1,\dots,n$$
are called  {\em generic elements of $A$}.

\end{definition}

Usually in the matrix algebra ${\mathbb M}_m(\k),$ the set of
matrix units $\{e_{ij}\}_{i,j=1}^m$ is taken as the base. In this
case $e_{ij}e_{kl}=\delta_{jk}e_{il}$ and
$z_l=\sum_{ij}\lambda_{ij}^le_{ij}$, $l=1,\dots,n$.

\begin{definition}
A   {\em generic matrix} is a matrix whose entries are distinct
   commutative indeterminates, and the so-called   {\em algebra of
generic matrices of order $m$} is generated by associative generic
$m\times m$ matrices.
\end{definition}

The algebra of generic matrices is prime, and every prime,
  relatively free, finitely generated associative
$PI$-algebra is isomorphic to an algebra of generic matrices. If
we include taking traces as an operator in the signature, then we
get the {\it algebra of generic matrices with trace}. That algebra
is a Noetherian module over its center.

Define the $\k$-linear mappings
$$
\Omega_i: \k_{\goth M}\!<\!\vec{x}\!>\to \k[\nu];\quad i=1,\dots,
n
$$
via the relation
$$
f(\sum_{i=1}^N\nu_{1i}e_i,\dots,\sum_{i=1}^N\nu_{ni}e_i)=
\sum_{i=1}^N (f\Omega_i)e_i.
$$
It is easy to see that the polynomials $f\Omega_i$ are uniquely determined
by $f$.

One can define the mapping
$$
\varphi_A: \End(k_{\goth M}\!<\!\vec{x}\!>)\to\End(k[\vec{\nu}])
$$
as follows: If
$$
\alpha\in\End(k_{\goth M}\!<\!\vec{x}\!>): x_s\to f_s(\vec{x})\quad
s=1,\dots,n
$$
then $\varphi_A(\alpha)\in\End(k[\vec{\nu}])$ can be defined via
the relation
$$
\varphi_A(\alpha): \nu_{si}\to P_{si}(\vec{\nu}); \quad s=1,\dots,n; \quad
i=1,\dots,n,
$$
where $P_{si}(\vec{\nu})=f_s\Omega_i$.

The following proposition is well known.

\begin{statement}[\cite{Yagzhev9}]
Let $A\in{\goth M}$ be a finite dimensional algebra, and
$\vec{x}=\{x_1,\dots,x_n\}$ be a finite set of commutative
indeterminates. Then the mapping $\varphi_A$ is a semigroup
homomorphism, sending $1$ to $1$, and automorphisms to
automorphisms. Also the mapping $\varphi_A$ commutes with the
operation $\widehat{} $ of taking the Jacobi endomorphism, in the
sense that $\widehat{\varphi_A(\alpha)}=\varphi_A(\hat{\alpha})$.
If $\varphi$ is invertible, then $\widehat{\varphi}$ is also
invertible.
%
\end{statement}

This proposition is important, since as noted after
Remark~\ref{chain}, the opposite direction is  the  $\JC$.

\subsection{Deformations and the Jacobian Conjecture for free associative
algebras}

\begin{definition} A {\it $T$-ideal} is a   completely characteristic ideal, i.e., stable under
any endomorphism.\end{definition}

\begin{statement}
Suppose $A$ is a relatively free algebra in the variety $\goth M$,
$I$~ is a $T$-ideal in $A$, and ${\goth M}'=\Var(A/I)$. Any
polynomial mapping $F: A\to A$
 induces  a natural mapping $F':A/I\to A/I$, as well as a
mapping $\widehat{F'}$ in~ ${\goth M}'$. If $F$ is invertible,
then $F'$ is invertible; if $\hat{F}$ is invertible, then
$\widehat{F'}$ is also invertible.
\end{statement}

For example, let $F$ be a polynomial endomorphism of the free
associative algebra $k\!<\vec{x}\!>$, and $I_n$ be the $T$-ideal
of the algebra of generic matrices of order $n$. Then
$F(I_n)\subseteq I_n$ for all $n$. Hence $F$ induces an
endomorphism $F_{I_n}$ of $k\!<\vec{x}\!>/I_n$. In particular,
this is a semigroup homomorphism. Thus, if $F$ is invertible, then
$F_{I_n}$ is invertible, but not vice versa.

The Jacobian mapping $\widehat{F_{I_n}}$ of the reduced
endomorphism $F_{I_n}$ is   the reduction of the Jacobian mapping
of $F$.

\subsubsection{The Jacobian Conjecture and the packing property}   \label{SbSbScPacking}

\indent   This subsection is based on the {\it packing property}
and deformations. Let us illustrate the main idea. It is well
known that the composite of ALL quadratic extensions of $\mathbb
Q$ is infinite dimensional over $\mathbb Q$.  Hence all such
extensions cannot be embedded (``packed'')  into a single
commutative finite dimensional $\mathbb Q$-algebra. However, all
of them can be packed into ${\mathbb M}_2(\mathbb Q)$. We
formalize the notion of packing in \S\ref{pack1}.  Moreover, for
ANY elements NOT in $\mathbb Q$ there is a parametric family of
embeddings (because it embeds non-centrally and thus can be
deformed via conjugation by a parametric set of matrices).
Uniqueness thus means belonging to the center. Similarly,
adjoining noncommutative coefficients allows one to decompose
polynomials, as to be elaborated below.

This idea allows us to solve equations via a finite dimensional
extension, and to find  a parametric sets of solutions if some
solution does not belong  to the original algebra. That situation
contradicts   local invertibility.

Let $F$ be an endomorphism of the free associative
algebra having invertible Jacobian.  We suppose that $F(0)=0$ and
$$F(x_i)=x_i+\sum\mbox{terms\ of}\ \mbox{order}\ge 2 .$$
We intend to show how the invertibility of the Jacobian implies
invertibility of the mapping $F$.

Let $ {Y}_1,\dots, {Y}_k$ be generic $m \times m$ matrices. Consider the
system of equations
$$
\left\{F_i( {X}_1,\dots, {X}_n)=Y_i; \quad i=1,\dots,k
\right\}.
$$
This system has a solution over some finite extension of order $m$
of the field generated by the center of the algebra of generic
matrices {\it with trace}.

Consider the set of block diagonal $mn\times mn$ matrices:

\begin{equation}        \label{EqMatrType}
A=\left(
\begin{array}{lcccl}
A_1&0& &\dots&0\cr 0&A_2&0&\dots&0\cr \vdots& &\ddots&
&\vdots \cr 0& &\dots& &A_n
\end{array}
\right),
\end{equation}
where the  $A_j$ are   $m\times m$ matrices.

Next, we  consider the system of equations
\begin{equation}\label{EqMatnk}
    \left\{F_i(X_1,\dots,X_n)=Y_i; \quad i=1,\dots,k \right\},
\end{equation}
where the $mn\times mn$ matrices $Y_i$ have the form
(\ref{EqMatrType})
with the $A_j$ generic matrices.

Any $m$-dimensional  extension of the base field $\k$ is embedded
into~$\mathbb{M}_m(\k)$. But $\mathbb{M}_{mn}(\k)\simeq
\mathbb{M}_{m}(\k)\otimes\mathbb{M}_{n}(\k)$. It follows that for
appropriate $m$, the system~(\ref{EqMatnk}) has a unique solution
in the matrix ring with traces. (Each    is given by a matrix
power series where the summands are matrices whose entries  are
homogeneous forms, seen by  rewriting   $Y_i  =  X_i +
∑\text{terms of order    2}$ as       $X_i  =  Y_i +
∑\text{terms of order  2}$, and iterating.)  The solution is
unique  since their entries are distinct commuting indeterminates.

If $F$ is invertible, then this solution must have   block
diagonal form. However, if $F$ is not invertible, this solution
need not have  block diagonal form.  Now we translate
invertibility of the Jacobian to the language of {\bf parametric
families} or {\bf deformations}.

Consider the matrices
$$
E_\lambda^\ell=
\left(
\begin{array}{lcccr}
E&0&& \dots&0\cr 0&\ddots&&\dots&0\cr 0&\dots&\lambda\cdot E&
&0\cr \vdots& &\dots&\ddots&\vdots\cr 0& &\dots&  &E
\end{array}
\right) $$ where $E$ denotes the identity matrix. (The index
$\ell$ designates the position of the  block $\lambda\cdot E$.)
Taking  $X_j$ not to be  a block diagonal matrix, then for some
$\ell$ we obtain a non-constant parametric family $E_\lambda^\ell
X_j(E_\lambda^\ell )^{-1}$ dependent on $\lambda$.

On the other hand, if $Y_i$ has form (\ref{EqMatrType})
then
$E_\lambda^\ell Y_i(E_\lambda^\ell)^{-1}=Y_i$ for all $\lambda\ne 0$; $ \ell=1,\dots,k$.

Hence, if $F_{I_n}$ is not an automorphism, then we have a {\bf
continuous parametric set of solutions}. But if the Jacobian
mapping is invertible, it is locally 1:1, a contradiction. This
argument yields the following result:

\begin{theorem}
For $F\in\End(k\!<\vec{x}\!>)$, if the Jacobian of $F$ is
invertible, then  the reduction $F_{I_n}$ of $F$, modulo the
$T$-ideal of the algebra of generic matrices, is invertible.
\end{theorem}

See \cite{Yagzhev8} for further details of the proof.
Because any  relatively free affine algebra of characteristic 0
  satisfies the set of identities of some
matrix algebra, it is the quotient of the algebra of generic
matrices by some $T$-ideal $J$. But $J$ maps into itself after any
endomorphism of the algebra. We conclude:

\begin{corollary}
If $F\in\End(k\!<\!\vec{x}\!>)$ and the Jacobian of $F$ is
invertible, then the reduction $F_{J}$ of $F$ modulo any proper
$T$-ideal $J$  is invertible.
\end{corollary}

In order to get invertibility of   $\vec F$ itself, Yagzhev used
the additional ideas:
\begin{itemize}
\item The block diagonal technique works equally well on  skew
fields. \item The above  algebraic constructions can be carried
out on Ore extensions, in particular for the {\it Weyl algebras}
$W_n=\k[x_1,\dots,x_n;\partial_1,\dots,\partial_n]$.

\item By a result of L.~Makar-Limanov, the free associative
algebra   can be embedded into the ring of fractions of the Weyl
algebra. This provides a nice presentation for mapping the free
algebra.
\end{itemize}


\begin{definition}
Let $A$ be an algebra, $B\subset A$ a subalgebra, and $\alpha:
A\to A$  a polynomial mapping of $A$ (and hence $\alpha(B)\subset
B$, see Definition \ref{DefPolMap}). $B$ is a {\it test algebra
for $\alpha$}, if $\alpha(A\backslash B)\ne A\backslash B$.
\end{definition}

The next theorem shows the universality of the notion of a test
algebra. An endomorphism is called {\it rationally invertible} if
it is invertible over {Cohn's}
 skew field of fractions~\cite{Cohn} of
$\k\!<\!\vec{x}\!>$.

\begin{theorem}[Yagzhev]
For any $\alpha\in \End(k\!<\vec{x}\!>)$, one of the two
statements holds:
\begin{itemize}
\item $\alpha$ is rationally invertible, and its reduction to any
finite dimensional factor  also is rationally invertible. \item
There exists a test algebra for some finite dimensional reduction
of $\alpha$.
\end{itemize}
\end{theorem}

This theorem implies the Jacobian conjecture for free associative
algebras. We do not go into details,
 referring the reader to the
papers \cite{Yagzhev8} and \cite{Yagzhev9}.

\medskip

{\bf Remark.} The same idea is used in quantum physics. The
polynomial $x^2+y^2+z^2$ cannot be decomposed for any commutative
ring of coefficients. However, it can decomposed using
noncommutative ring of coefficients (Pauli matrices). The Laplace
operator in $3$-dimensional space can be decomposed in such a
manner.

\subsubsection{Reduction to nonzero characteristic}

 One can work with deformations equally well in nonzero
characteristic. However, the naive Jacobian condition does not
give us parametric families, because of
 consequences  of inseparability. Hence it is interesting using
deformations to get a reasonable version of the $\JC$
 for characteristic $p> 0$, especially because of recent
progress in the    $\JC$  related to the reduction of holonomic
modules to the case of characteristic $p$ and investigation of the
$p$-curvature or Poisson brackets on the center
\cite{BelovKontsevich1}, \cite{BelovKontsevich}, \cite{Tsuch}.

In his very last paper \cite{YagzhevLast} A.V.~Yagzhev approached
the    $\JC$ using positive characteristics. He noticed that the
existence of a counterexample
 is
equivalent to the existence of an Engel, but not Yagzhev, finite
dimensional ternary algebra in each positive characteristic $p\gg
0$. (This fact is also used in the papers
\cite{BelovKontsevich1,BelovKontsevich,Tsuch}.)

If a counterexample to the    $\JC$ exists, then  such an algebra
$A$ exists  even over a finite field, and hence  can be finite. It
generates a locally finite variety of algebras that are of Engel
type, but not Yagzhev. This situation can be reduced to the case
of a locally semiprime variety. Any relatively free algebra of
this variety is semiprime, and the centroid of its localization is
a finite direct sum of fields. The situation can be reduced to one
field, and he tried to construct an embedding which is not an
automorphism. This would contradict the finiteness property.

Since a reduction of an endomorphism as a mapping on points of
finite height may be an automorphism,  the issue of injectivity
also arises. However, this approach looks promising, and may
involve new ideas, such as  in the papers
\cite{BelovKontsevich1,BelovKontsevich,Tsuch}. Perhaps different
infinitesimal conditions (like the Jacobian condition in
characteristic zero) can be found.

\subsection{The Jacobian Conjecture for other classes of
algebras} Although the Jacobian Conjecture remains open for
commutative associative algebras, it has been established for
other classes of algebras, including free associative algebras, free Lie algebras, and free metabelian algebras. See \S\ref{SbScJCArbVar} for further details.

 An
algebra is {\it metabelian} if it satisfies the identity
$[x,y][z,t]=0$.

The case of free metabelian algebras, established by Umirbaev
\cite{UmirbaevMetabelean}, involves some interesting new ideas
that we describe now. His method of proof is by means of
co-multiplication, taken from the theory of Hopf algebras and
quantization. Let $A^{op}$ denote the opposite algebra of  the
free associative algebra $A$, with generators $t_i$. For $f\in A$
we denote the corresponding element of $A^{op}$ as  $f^*$. Put
$\lambda: A^{op}\otimes A\to A$ be the mapping such that
$\lambda(\sum f_i^* \otimes g_i)=\sum f_ig_i$.
$I_A:=\ker(\lambda)$ is a free $A$ bimodule with generators
$t_i^*\otimes 1-1\otimes t_i$. The mapping $d_A: A\to I_A$ such
that $d_A(a)=a^*\otimes 1-1\otimes a$ is called the {\it universal
derivation} of $A$. The {\it Fox derivatives} $\partial a/\partial
t_i\in A^{op}\otimes A$ \cite{F} are defined via $d_A(a)=\sum_i
(t_i^*\otimes 1-1\otimes t_i)\partial a/\partial t_i$,
cf.~\cite{DicksLevin} and~\cite{UmirbaevMetabelean}.

Let $C=A/\Id([A,A])$, the free commutative associative algebra,
and let $B=A/\Id([A,A])^2$, the free metabelian algebra. Let
$$\partial(a)=(\partial a/\partial t_1,\dots,\partial a/\partial
t_n).$$ One can define the natural derivations

$$\bar{\partial}:A\to (A'\otimes A)^n\to(C'\otimes C)^n,$$

\begin{equation}\label{nat} \tilde{\partial}:A\to (C'\otimes C)^n\to C^n.\end{equation}
where the mapping $(C'\otimes C)^n)\to C^n$ is induced by
$\lambda$.  Then $\ker(\bar{\partial})=\Id([A,A])^2+F$ and
$\bar{\partial}$ induces a derivation $B\to (C'\otimes C)^n$,
whereas $\tilde{\partial}$ induces the usual derivation $C\to
C^n$. Let $\Delta: C\to C'\otimes C$ be the mapping induced by
$d_A$, i.e., $\Delta(f)=f^*\otimes 1-1\otimes f$, and let
$z_i=\Delta(x_i)$.
The {\it Jacobi
matrix} is defined in the  natural way, and provides the formulation
of the
 $\JC$ for free metabelian algebras. One of the crucial steps in
proving the~$\JC$ for free metabelian algebras is the following
homological lemma from~\cite{UmirbaevMetabelean}:

\begin{lemma}
Let $\vec{u}=(u_1,\dots,u_n)\in (C^{op}\otimes C)^n$. Then
$\vec{u}=\bar{\partial}(\bar{w})$ for some $w\in\Id([A,A])$ iff
$$\sum z_iu_i=0.$$
\end{lemma}

The proof also requires the following theorem:

\begin{theorem}
Let $\varphi\in\End(C)$. Then $\varphi\in\Aut(C)$ iff $\Id(\Delta(
\varphi(x_i) ))_{i=1}^n=\Id(z_i)_{i=1}^n$.
\end{theorem}

The paper \cite{UmirbaevMetabelean} also includes the following
result:

\begin{theorem}  \label{ThBClifting}
Any automorphism of $C$ can be extended to an automorphism of $B$,
using the $\JC$ for the free metabelian algebra $B$.
\end{theorem}

This is a nontrivial result, unlike the extension of an automorphism
of $B$ to an automorphism of $A/\Id([A,A])^n$ for any $n>1$.
%
%
%
%
%
%
%
%
%

\subsection{Questions related to the Jacobian Conjecture}                            \label{SbScRelQuestJC}
Let us turn to other interesting questions which can be linked to
the Jacobian Conjecture. The quantization procedure is a bridge
between the commutative and noncommutative cases and is deeply
connected to the $\JC$ and related questions. Some of these
questions also are discussed in the paper \cite{DrYuLift}.

Relations between the free associative algebra and the classical
commutative situation are very deep. In particular, Bergman's
theorem that any commutative subalgebra of the free associative
algebra is isomorphic to a polynomial ring in one indeterminate is
the noncommutative analog of Zak's theorem \cite{Zaks} that any
integrally closed subring  of a polynomial ring of Krull dimension
$1$ is isomorphic to a polynomial ring in one indeterminate.

For example, Bergman's theorem is used to describe the
automorphism group $\Aut(\End(\k\langle x_1,\dots,x_n\rangle))$
\cite{BelovLipBer}; Zak's theorem is used in the same way to
describe the group $\Aut(\End(\k[x_1,\dots,x_n]))$
\cite{BelovLiapiansk2}.

\medskip
{\bf Question.}\ {\it Can one prove Bergman's theorem via
quantization?}
\medskip

Quantization could be a key idea for understanding Jacobian type
problems in  other varieties of algebras.

\paragraph{1. Cancellation problems.}$ $

\medskip

We recall three classical problems.

\medskip

  {\bf 1.} {\it Let $K_1$ and $K_2$ be affine domains for
which $K_1[t]\simeq K_2[t]$. Is it true that $K_1\simeq K_2$? }

  {\bf 2.} {\it Let $K_1$ and $K_2$ be an affine fields
for which $K_1(t)\simeq K_2(t)$. Is it true that $K_1\simeq K_2$?
In particular, if $K(t)$  is a field of rational functions  over
the field~$\k$, is it true that $K$ is also a field of rational
functions over $\k$? }

  {\bf 3.} {\it If $K[t]\simeq \k [x_1,\dots,x_n]$, is it
true that $K\simeq \k [x_1,\dots,x_{n-1}]$? }
\medskip

The answers to Problems 1 and 2 are `No' in general (even if $\k
={\mathbb C}$); see the fundamental paper~\cite{BCSS}, as well as
\cite{BYML} and the references therein. However,  Problem 2 has a
positive solution in low dimensions. Problem 3 is currently called
the {\it Cancellation Conjecture}, although Zariski's original
cancellation conjecture was for fields (Problem~2). See
(\cite{MiySugie}, \cite{KZ}, \cite{Dan}, \cite{ShYu1}) for
Zariski's conjecture and related problems. For $n\ge 3$, the
Cancellation Conjecture (Problem 3) remains open, to the best of
our knowledge, and it is reasonable to pose the Cancellation
Conjecture for free associative rings and ask the following:

\medskip
\noindent {\bf Question.} {\it If $K*\k [t]\simeq \k
\!\!<\!x_1,\dots,x_n\!>$, then is $K\simeq \k
\!<\!x_1,\dots,x_{n-1}\!>? $}
\medskip

This question was solved  for $n=2$  by V.~Drensky and J.T.~Yu
\cite{DrenskyYu}.

\paragraph{2. The Tame Automorphism Problem.} Yagzhev utilized his
approach to study the tame automorphism problem. Unfortunately,
these  papers are not preserved.

It is easy to see that every endomorphism $\phi$ of a commutative
algebra can be lifted to some endomorphism of the free associative
algebra, and hence to some endomorphism of the algebra of generic
matrices. However, it is not clear that any automorphism $\phi$
can be lifted to an automorphism.

We recall that an automorphism  of $\k [x_1,\dots,x_n]$ is  {\em elementary} if it has the form
$$x_1 \mapsto x_1 + f(x_2, \dots, x_n), \quad x_i \mapsto x_i, \quad \forall i \ge 2.$$
 A  {\em tame automorphism} is a product of elementary automorphisms, and a non-tame
automorphism is called  {\em wild}. The ``tame automorphism
problem'' asks whether any automorphism is tame. Jung~\cite{J} and
van der Kulk \cite{vdK} proved this for $n=2$, (also see
\cite{Nie1,Nie2} for free groups, \cite{Cohn} for free Lie
algebras, and \cite{ML,Cz} for free associative algebras), so one
takes $n>2.$

 Elementary automorphisms can be
lifted to automorphisms of the free associative algebra; hence
every tame automorphism can be so lifted. If an automorphism
$\varphi$ cannot be lifted to an automorphism of the algebra of
generic matrices, it cannot be tame. This give us approach to the
tame automorphism problem.

We can lift an automorphism of $\k [x_1,\dots,x_n]$ to an
endomorphism of $\k \!\!<x_1,\dots,x_n\!>$ in many ways. Then
replacing $x_1,\dots,x_n$ by $N\times N$ generic matrices induces
a polynomial mapping $F_{(N)}: \k ^{nN^2}\to \k ^{nN^2}$.

For each automorphism $\varphi$, the invertibility of this mapping
can be transformed into  compatibility   of some system of
equations. For example, Theorem~10.5 of \cite{Peretz} says that
the Nagata automorphism is wild, provided that a certain system of
five equations in $27$ unknowns has no solutions. Whether Peretz'
method  can effectively attack tameness questions remains to be
seen. The wildness of the Nagata automorphism was established by
Shestakov and Umirbaev~\cite{ShUm}. One important ingredient in the proof is
 {\it degree estimates} of an expression $p(f,g)$ of algebraically independent polynomials $f$ and $g$ in terms of the degrees of $f$ and $g$,
 provided neither leading term is proportional to a power of the other, initiated by Shestakov and Umirbaev~\cite{ShUm0}.  An exposition based on their
method is given in Kuroda \cite{Kuroda}.

One of the most important tools is the degree estimation technique, which
in the multidimensional case becomes the analysis of leading
terms, and is more complicated.
We refer to the deep papers
\cite{BonnetVenerau,Kuroda,Marek}. Several papers of Kishimoto
 contain gaps, but also provide deep insights.

One can also ask the weaker question of ``coordinate tameness:''
 Is the image of $(x,y,z)$ under the Nagata automorphism the image  under some
 (other) tame automorphism? This also fails, by
 \cite{UY}.

An automorphism $\varphi$ is called {\em stably tame} if,  when
several new indeterminates~$\{t_i\}$ are adjoined, the extension of
$\varphi$ given by  $\varphi'(t_i)=t_i$ is tame; otherwise it is
called {\em stably wild}. Stable tameness of automorphisms of $\k
[x,y,z]$ fixing~$z$ is proved in \cite{BEW}; similar results for
$\k \langle x,y,z\rangle$ are given in \cite{BKYStbTmn}.

Yagzhev tried to construct wild automorphisms via polynomial
automorphisms of the Cayley-Dickson algebra with base
$\{\vec{e}_i\}_{i=1}^8$, and the set $\{\nu_i, \xi_i,
\varsigma_i\}_{i=1}^8$ of commuting indeterminates. Let
 $$x=\sum \nu_i\vec{e}_i,\ y=\sum \xi_i\vec{e}_i,\ z=\sum
 \varsigma_i\vec{e}_i.$$
 Let $(x,y,z)$ denote the associator $(xy)z - x (yz)$ of the elements $x,y,z$,
 and write
 $$(x,y,z)^2=\sum f_i(\vec{\nu},\vec{\xi},\vec{\varsigma})\vec{e}_i.$$

Then the endomorphism $G$ of  the polynomial algebra given by
 $$G:\
 \nu_i\to\nu_i+f_i(\vec{\nu},\vec{\xi},\vec{\varsigma}),\quad
 \xi_i\to\xi_i,\ \varsigma_i\to\varsigma_i,$$
is an automorphism, which likely is stably wild.

In the free associative case, perhaps it is possible to construct
an example of an automorphism, the wildness of which could be
proved by considering its Jacobi endomorphism
(Definition~\ref{Jac1}). 
Yagzhev tried to construct
 examples of algebras $R=A\otimes A^{op}$ over which there are invertible
matrices that cannot decompose as products of elementary ones.
Yagzhev conjectured that the automorphism
 $$x_1\to x_1+y_1(x_1y_2-y_1x_2),\ x_2\to x_2+(x_1y_2-y_1x_2)y_2,\
y_1\to y_1,\ y_2\to y_2$$
 of the free associative algebra is wild.

 Umirbaev~\cite{UmirbaevAnic}  proved  in characteristic 0 that
the {\it Anick automorphism} $x\to x + y(xy-yz)$, $y\to y$, $z\to
z+(zy-yz)y$ is wild, by using metabelian algebras. The proof uses
description of the defining relations of 3-variable automorphism
groups~\cite{U4,U5,U6}. Drensky and  Yu~\cite{DY3,DY4} proved in characteristic 0 that
the image of $x$ under the Anick Automorphism is not the
image of any tame automorphism.

\medskip
{\bf Stable Tameness Conjecture}. {\it Every automorphism of the   polynomial
 algebra $\k [x_1, \dots, x_n]$, resp.~of the free  associative algebra $\k \!\!<\!x_1 ,\dots, x_n\!>$, is stably tame}.
\medskip

Lifting in the free associative case is related to quantization.
It provides some light on the similarities and differences between
the commutative and noncommutative cases. Every tame automorphism
of the polynomial ring can be lifted to an automorphism of the
free associative algebra. There was a conjecture that {\it any
wild $z$-automorphism of $\k [x,y,z]$ (i.e., fixing $z$) over an
arbitrary field $\k$ cannot be
lifted to a $z$-automorphism of $\k\!\!<\!x,y,z\!>$.} In particular, the
Nagata automorphism cannot be so lifted \cite{DrYuLift}. This
conjecture was solved by Belov and J.-T.Yu \cite{BKY} over an arbitrary field. However,
the general lifting conjecture is still open. In particular, it is
not known whether the Nagata automorphism can be lifted to an
automorphism of the free algebra. (Such a lifting could not fix
$z$).

The paper \cite{BKY}  describes all the $z$-automorphisms of $\k
\!\!<\!x,y,z\!>$ over an arbitrary field  $\k $.
Based on that work, Belov and J.-T.Yu \cite{BKYStbTmn} proved that every $z$-automorphism
of $\k
\!\!<\!x,y,z\!>$  is stably tame, for all fields~$\k $. A~similar result in the commutative case is proved by Berson, van den Essen, and Wright~\cite{BEW}. These are    important first steps towards solving the stable tameness conjecture in the noncommutative and commutative cases.

The free associative situation is much more rigid than the
polynomial case. Degree estimates for the free associative case
are the same for prime
  characteristic   \cite{YuYungChang} as in characteristic~0~\cite{MLY}. The methodology is  different
 from the commutative case,
 for which degree estimates (as well as
examples of wild automorphisms) are not known in prime characteristic.

J.-T.Yu found some evidence of a connection between the lifting
conjecture and the Embedding Conjecture of Abhyankar and Sathaye.
Lifting seems to be ``easier''.

\subsection{Reduction to simple algebras}\label{pack1}
This subsection is devoted to finding test algebras.

Any prime algebra $B$   satisfying a system of Capelli identities
of order $n+1$ ($n$ minimal such) is said to have {\it rank} $n$.
In this case,  its operator algebra is PI.
The localization of $B$ is a
simple algebra of dimension $n$ over its centroid, which is a
field. This is the famous {\it rank theorem} \cite{RazmyslovBook}.

\subsubsection{Packing properties}\label{pack1}

\begin{definition}
Let ${\cal M}=\{{\goth M}_i: i \in I\}$ be an arbitrary set of
varieties of algebras. We say that $\cal M$ satisfies the {\em
packing property}, if for any $n\in{\mathbb N}$ there exists a
prime algebra $A$ of rank $n$ in some ${\goth M}_j$ such that any
prime algebra in any~${\goth M}_i$ of rank $n$ can be embedded
into some central extension  $K\otimes A$ of $A$.

$\cal M$ satisfies the {\em  finite packing property}  if, for any
finite set of prime algebras $A_j\in{\goth M}_i$, there exists a
prime algebra $A$ in some ${\goth M}_k$ such that each $A_j$ can
be embedded into $A$.
\end{definition}

The set of proper subvarieties of associative algebras satisfying
a system of Capelli identities of some order $k$  satisfies the
packing property (because any simple associative algebra is a
matrix algebra over field).

 However,   the  varieties of
alternative algebras satisfying a system of Capelli identities of
order $>8$,   or of Jordan algebras satisfying a system of
Capelli identities of order $>27$,    do not even satisfy the
finite packing property. Indeed, the matrix algebra of order $2$
and the Cayley-Dickson algebra cannot be embedded into a common
prime alternative algebra. Similarly, ${\mathbb H}_3$ and the
Jordan algebra of symmetric matrices cannot be embedded into a
common Jordan prime algebra. (Both of these assertions follow
easily by considering their PIs.)

It is not known whether or not  the packing property holds for
Engel algebras satisfying a system of Capelli identities; knowing
the answer would enable us to resolve the  $\JC$, as will be seen
below.

\begin{theorem}
If the set of varieties of Engel algebras (of arbitrary fixed
order) satisfying a system of Capelli identities of some order
satisfies the packing property, then the Jacobian Conjecture has a
positive solution.
\end{theorem}

\begin{theorem}\label{ThWeakStsfy}
The set of varieties from the previous theorem satisfies the
finite packing property.
\end{theorem}

Most of the remainder of this section is devoted to the proof of
these two theorems.

\medskip
{\bf Problem.} {\it Using  the packing property and deformations,
give a reasonable analog of the $\JC$ in nonzero characteristic.}
(The naive approach using only  the determinant of the Jacobian
does not work.)
\medskip

 \subsubsection{Construction of simple Yagzhev algebras}   

Using the Yagzhev correspondence and
composition of elementary automorphisms it is possible to
construct a new algebra of Engel type.

\begin{theorem}
Let $A$ be an algebra of Engel type. Then $A$ can be embedded into
a prime algebra of Engel type.
\end{theorem}

\Proof Consider the  mapping $F:V\to V$  (cf.~\eqref{Eq1Fk}) given by
$$F:\
x_i\mapsto x_i+\sum_j\Psi_{ij};\quad i=1,\dots,n$$ (where the
$\Psi_{ij}$ are forms of homogenous degree $j$). Adjoining new
indeterminates $\{t_i\}_{i=0}^n$, we put $F(t_i)=t_i$ for $
i=0,\dots,n$.

Now we  take the transformation
$$
G:\ t_0\mapsto t_0,\quad x_i\mapsto x_i,\quad  \ t_i\mapsto
t_i+t_0x_i^2,\quad\mbox{for}\quad  i=1,\dots,n.
$$
The composite $F\circ G$ has invertible Jacobian (and hence the
corresponding algebra has Engel type) and can be expressed as
follows:
$$
F\circ G:\ x_i\mapsto x_i+\sum_j\Psi_{ij},\quad \ t_0\mapsto
t_0,\quad t_i\mapsto t_i+t_0x_i^2\quad \mbox{for}\quad i=1,\dots,n.
$$
It is easy to see that the corresponding algebra $\widehat{A}$ also
satisfies the following properties:

\begin{itemize}
\item $\widehat{A}$ contains $A$ as a subalgebra (for $t_0=0$).
\item If $A$ corresponds to a cubic homogenous mapping (and thus
is Engel) then $\widehat{A}$ also corresponds to a cubic
homogenous mapping (and thus is Engel). \item If some of the forms
$\Psi_{ij}$ are not zero, then $A$ does not have nonzero ideals
with product 0, and hence is prime (but its localization need not
be simple!).
\end{itemize}

Any algebra $A$ with  operators  can be embedded, using the
previous construction, to a prime algebra with nonzero
multiplication.   The theorem is proved. \Endproof

Embedding via the previous theorem preserves the cubic homogeneous
case, but does not yet  give us an embedding into a simple algebra
of Engel type.

\begin{theorem}\label{3.8}
Any algebra $A$   of Engel type  can be embedded into a simple
algebra  of Engel type.
\end{theorem}

\Proof We  start from the following observation:

\begin{lemma}
Suppose $A$ is a finite dimensional algebra, equipped with a base
$\vec{e}_1,\dots,\vec{e}_n,\vec{e}_{n+1}$. If for any $1\le i,
j\le n+1$ there exist operators $\omega_{ij}$ in the signature
$\Omega(A)$ such that
$\omega_{ij}(\vec{e}_i,\dots,\vec{e}_i,\vec{e}_{n+1})=\vec{e}_j$,
with all other values on the base vectors being zero, then $A$ is
simple.
\end{lemma}

This lemma implies:

\begin{lemma}
Let $F$ be a polynomial endomorphism of ${\mathbb
C}[x_1,\dots,x_n;t_1,t_2]$, where $$F(x_i)=\sum_j \Psi_{ij}.$$ For
notational convenience we put $x_{n+1}=t_1$ and $x_{n+2}=t_2$. Let
$\{k_{ij}\}_{i=1,j}^s$ be a set of natural numbers such that
\begin{itemize}
\item For any $x_i $ there exists $k_{ij}$ such that among all
$\Psi_{ij}$ there is {\em exactly one} term of degree $k_{ij}$,
and it has the form $\Psi_{i,k_{ij}}=t_1x_j^{k_{ij}-1}$.

\item For $t_2$ and any $x_i$ there exists $k_{iq}$ such that
among all $\Psi_{ij}$ there is {\em exactly one} term of degree
$k_{iq}$, and it has the form
$\Psi_{n+2,k_{iq}}=t_1x_j^{k_{iq}-1}$.

\item For $t_1$ and any $x_i$ there exists $k_{iq}$ such that among
all $\Psi_{ij}$ there is {\em exactly one} term of degree
$k_{iq}$, and it has the form
$\Psi_{n+1,k_{iq}}=t_2x_j^{k_{iq}-1}$.
\end{itemize}
Then the corresponding algebra is simple.

\end{lemma}

\Proof Adjoin the term $t_\ell x_i^{k-1}$   to the $x_i$, for $\ell = 1,2$. Let $e_i$
  be the base vector  corresponding to $x_i$. Take the
  corresponding
$k_{ij}$-ary operator $$\omega:
\omega(\vec{e}_i,\dots,\vec{e}_i,\vec{e}_{n+\ell}))=\vec{e}_j,$$ with
all other products   zero. Now we apply the previous lemma.
 \Endproof

\medskip
{\bf Remark.} In order to be flexible with constructions via the
Yagzhev correspondence, we are working in the general, not
necessary cubic, case.
\medskip

Now we can conclude the proof of Theorem~\ref{3.8}. Let $F$ be the
mapping corresponding to the algebra $A$:

$$F:\
x_i\mapsto x_i+\sum_j\Psi_{ij},\quad i=1,\dots,n,$$ where
$\Psi_{ij}$ are forms of homogeneous degree $j$. Let us adjoin new
indeterminates $\{t_1, t_2\}$ and put $F(t_i)=t_i,$ for $ i=1, 2$.

We choose all $ k_{\alpha,\beta}>\max(\deg(\Psi_{ij}))$ and assume
that these numbers are sufficiently large. Then we consider the
mappings
$$
G_{k_{ij}}: x_i\mapsto x_i+x_j^{k_{ij}-1}t_1,\ i\le n;\quad
t_1\mapsto t_1;\quad t_2\mapsto t_2;\quad x_s\mapsto x_s\
\mbox{for}\ s\ne i.
$$
$$
G_{k_{i(n+2)}}: t_2\mapsto x_i^{k_{ij}-1}t_1;\quad t_1\mapsto
t_1;\quad x_s\mapsto x_s\ \mbox{for}\  1\le s\le n.
$$
$$
G_{k_{i(n+1)}}: t_1\mapsto x_i^{k_{ij}-1}t_2;\quad t_2\mapsto
t_2;\quad x_s\mapsto x_s\ \mbox{for}\  1\le s\le n.
$$
These mappings are elementary automorphisms.

Consider the mapping $H=\circ_{k_{ij}}G_{k_{ij}}\circ F$, where
the composite is taken in order of ascending $k_{\alpha\beta}$,
and then with $F$.
 If the $k_{\alpha\beta}$ grow quickly enough, then the terms
 obtained in the previous step do not affect the
lowest term obtained at the next step, and this term will be as
described in the lemma.  The theorem is proved. \Endproof

\medskip
{\bf Proof of Theorem \ref{ThWeakStsfy}.} The direct sum of Engel
type algebras  is also of Engel type, and by Theorem~\ref{3.8} can
be embedded
 into a simple algebra of Engel type.
\Endproof

\paragraph{The Yagzhev correspondence and algebraic extensions.}$ $

For notational simplicity, we  consider a cubic homogeneous
mapping
$$
F:\ x_i\mapsto x_i+\Psi_{3i}(\vec{x}).
$$
We shall construct the Yagzhev correspondence of an algebraic
extension.

Consider the equation
$$
t^s=\sum_{p=1}^s\lambda_pt^{s-p},
$$
where the $\lambda_p$ are formal parameters. If $m\ge s$, then for
some $\lambda_{pm}$, which  can be expressed as polynomials in
$\{\lambda_p\}_{p=1}^{s-1}$,  we have
$$
t^m= \sum_{p=1}^s\lambda_{pm}t^{s-p}.
$$

Let $A$ be the algebra corresponding to the mapping $F$. Consider
$$A\otimes \k[\lambda_1,\dots,\lambda_s]$$ and its finite algebraic
extension $\hat{A}=A\otimes \k[\lambda_1,\dots,\lambda_s,t]$. Now
we take the mapping corresponding (via the Yagzhev correspondence)
to the ground ring $R=\k [\lambda_1,\dots,\lambda_s]$ and algebra
$\hat{A}$.

For $m=1,\dots,s-1$, we define new formal indeterminates, denoted
as~$T^mx_i$.  Namely, we put $T^0x_i=x_i$ and for $m\ge s$, we
identify $ T^mx_i$ with $ \sum_{p=1}^s\lambda_{pm}T^{s-p}x_i $,
where $\{\lambda_p\}_{p=1}^{s-1}$ are formal parameters in the
 centroid of some extension $R\otimes A$.
Now we extend the mapping~$F$, by putting
$$
F(T^mx_i)=T^mx_i+T^{3m}\Psi_{3i}(\vec{x}),\quad m=1,\dots,s-1.
$$
We get a natural mapping corresponding to the algebraic
extension.

Now we can take more symbols $T_j$, $ j=1,\dots, s$, and equations
$$
T_j^s=\sum_{p=1}^s\lambda_{pj}T_j^{s-p}
$$
and a new set of indeterminates $x_{ijk}=T_j^kx_i$ for $
j=1,\dots,s$ and $ i=1,\dots, n$. Then we put
$$
x_{ijm}=T_j^mx_i= \sum_{p=1}^s\lambda_{jpm}T_j^{s-p}x_i
$$
and
$$
F(x_{ijm})=x_{ijm}+T_j^{3m}\Psi_{3i}(\vec{x}),\quad m=1,\dots,s-1.
$$
This yields an ``algebraic extension'' of $A$.

\paragraph{Deformations of algebraic extensions.} Let $m=2$. Let us
introduce new indeterminates $y_1,y_2$, put $F(y_i)=y_1, \ i=1,2$,
and compose $F$ with the automorphism $$G: T_1^1x_i\mapsto
T_1^1x_i+y_1x_i, \quad T_1^1x_i\mapsto T_2^1x_i+y_1x_i, \quad x_i \mapsto x_i,    \quad i = 1,2,$$ $$ y_1\mapsto
y_1+y_2^2y_1,\quad  y_2\mapsto y_2. $$ (Note that the $T_1^1x_i$
and $T_2^1x_i$ are {\it new} indeterminates and not proportional
to $x_i$!) Then compose $G$ with the automorphism $H: y_2\mapsto
y_2+y_1^2$, where $H$  fixes the other indeterminates. Let us call
the corresponding new algebra $\hat{A}$. It is easy to see that
$\Var(A)\ne\Var(\hat{A})$.


 Define an  {\it identity of the pair} $(A,B)$, for
$A\subseteq B$ to be a polynomial in two sets of indeterminates
$x_i, z_j$ that vanishes whenever the $x_i$ are evaluated in $A$
and $z_j$ in $B$.) The {\it variety of the pair} $(A,B)$ is the
class of  pairs of algebras   satisfying the identities of
$(A,B)$.

 Recall that by  the rank theorem, any prime algebra $A$ of rank $n$
can be embedded into an $n$-dimensional simple algebra $\hat{A}$.
We consider the variety of the  pair $(A,\hat{A})$.

Considerations of deformations yield the following:

\begin{statement}
Suppose for all simple $n$-dimensional pairs there exists a
universal pair in which all of them can be embedded. Then the
Jacobian Conjecture has a positive solution.
\end{statement}

We see the relation with

\medskip
{\bf The Razmyslov--Kushkulei theorem}~\cite{RazmyslovBook}: \ {\it Over an
algebraically closed field, any two finite dimensional simple
algebras satisfying the same identities are isomorphic.}
\medskip

The difficulty in applying this theorem is that the identities may
depend on parameters. Also, the natural generalization of the
Rasmyslov--Kushkulei theorem for a variety and subvariety does not
hold: Even if $\Var(B)\subset\Var(A)$, where $B$ and $A$ are
simple finite dimensional algebras over some algebraically closed
field,
  $B$ need not be embeddable to $A$.
%

\bigskip

{\bf Addresses}

\medskip

 \noindent
Alexei Belov: Bar-Ilan University

\smallskip

\noindent {\it e-mail address\/}: ~beloval@math.biu.ac.il

\medskip

\noindent Leonid Bokut: Sobolev Institute of Mathematics,
Novosibirsk,

South China Normal University, Guangzhou

\smallskip

\noindent {\it e-mail address\/}:~bokut@math.nsc.ru

\medskip

\noindent Louis Rowen: Bar-Ilan University

\noindent {\it e-mail address\/}: ~rowen@math.biu.ac.il

\medskip

\noindent Jie-Tai Yu: Department of Mathematics, The University of
Hong Kong, Hong Kong SAR, China

\noindent {\it e-mail address\/}: yujt@hku.hk,\ yujietai@yahoo.com

\bigskip

\end{document}